\documentclass[11pt,reqno]{amsart}
\usepackage{amssymb,amsmath}\newcommand{\be}{\begin{eqnarray}}
\newcommand{\ee}{\end{eqnarray}}\newcommand{\card}{\#}\newcommand{\eps}{{\mbox{$\epsilon$}}}\newcommand{\e}{{\varepsilon}}\newcommand{\R}{{\mathbb R}}\newcommand{\Z}{{\mathbb Z}}\newcommand{\LL}{\mathcal L}\newcommand{\K}{{\mathcal K}}

\newcommand{\dist}{\operatorname{dist}}\newcommand{\Fav}{\operatorname{Fav}}\newcommand{\supp}{\operatorname{supp}}\newtheorem{theorem}{Theorem}\newtheorem{lemma}[theorem]{Lemma}\newtheorem{cor}[theorem]{Corollary}\newtheorem{prop}[theorem]{Proposition}\theoremstyle{definition}\theoremstyle{remark}\numberwithin{equation}{section}\input epsf.sty
\begin{document}\thispagestyle{empty}

%
%
%
%
%
%
%
%
%
%
%
%
\newcommand{\G}{{\mathcal G}}
\newcommand{\C}{{\mathbb C}}
\newcommand{\Pa}{{P_{1,\theta}(y)}}
\newcommand{\Pb}{{P_{2,\theta}(y)}}
\newcommand{\Pshp}{{P_{1,t}^\sharp(x)}}
\newcommand{\Pflt}{{P_{1,t}^\flat(x)}}
\newcommand{\OTL}{{\Omega(\theta,\ell)}}
\newcommand{\rsz}{{R(\theta^*)}}
\newcommand{\phitil}{{\tilde{\varphi}_t}}
\newcommand{\lambdatil}{{\tilde{\lambda}}}
\newcommand{\Sier}{{\mathcal S}}

\title[Buffon's needle and Besicovitch irregular self-similar sets]{{Buffon's needle landing near Besicovitch irregular self-similar sets}}
\author{Matthew Bond}\address{Matthew Bond, Dept. of Math., Michigan State University.
{\tt bondmatt@msu.edu}}
\author{Alexander Volberg}\address{Alexander Volberg, Dept. of  Math., Michigan State Univ. 
and the University of Edinburgh.
{\tt volberg@math.msu.edu}}

\thanks{Research of the authors was supported in part by NSF grants  DMS-0501067, 0758552 }

\begin{abstract} In this paper we get an upper estimate of the Favard length of an arbitrary neighborhood of an arbitrary self-similar Cantor set.
Consider $L$ closed discs of radius $1/L$ inside the unit disc. By using linear maps of the disc onto the smaller discs we can generate a self-similar Cantor set $\G$. One such process is to let $\G_n$ be the union of all possible images of the unit disc under $n$-fold compositions of the similarity maps. Then $\G=\bigcap_n\G_n$.
One may then ask the rate at which the Favard length -- the average over all directions of the length of the orthogonal projection onto a line in that direction -- of these sets $\G_n$ decays to zero as a function of $n$. Previous quantitative results for the Favard length problem were obtained by Peres--Solomyak \cite{PS} and Tao \cite{T}; in the latter paper a general way of making a quantitative statement from the Besicovitch theorem is considered. But being rather general, this method does not give a good estimate for self-similar structures such as  $\G_n$.
In the present work we prove the estimate $\Fav(\G_n)\leq e^{-c\sqrt{\log\,n}}$. While this estimate is vastly improved compared to \cite{PS} and \cite{T}, it is worse than the power estimate $\Fav(\G_n)\leq\frac{C}{n^p}$ proved for specific sets $\G_n$ with additional product structures in Nazarov-Peres-Volberg \cite{NPV} and Laba-Zhai \cite{LZ}. The power estimate still appears to be related to a certain regularity property of zeros of a corresponding linear combination of exponents (we call this property {\it analytic tiling}). We consider also the Sierpinski gasket, where this regularity of zeros exists, resulting in an improvement to a power estimate.
\end{abstract}
\maketitle

\section{{\bf Introduction}} \label{sec:intro}

Let $E\subset\C$, and let $\text{proj}_\theta$ denote orthogonal projection onto the line having angle $\theta$ with the real axis. The \textbf{average projected length} or \textbf{Favard length} of $E$, $\text{Fav}(E)$, is given by
$$\text{Fav}(E)=\frac1{\pi}\int_0^{\pi}|\text{proj}_\theta(E)|d\,\theta.$$
For bounded sets, Favard length is also called \textbf{Buffon needle probability}, since up to a normalization constant, it is the likelihood that a long needle dropped with independent, uniformly distributed orientation and distance from the origin will intersect the set somewhere.

Consider $L$ closed discs of radius $1/L$ inside the unit disc. By using linear maps of the unit disc onto the smaller discs we can generate a self-similar Cantor set $\G$. A partial construction $\G_n$ of $\G$ consists of the union of all possible images of the unit disc under $n$-fold compositions of these similarity maps. Then $\G=\bigcap_n\G_n$.

One may then ask the rate at which the Favard length of these sets $\G_n$  decays to zero as a function of $n$(\footnote{Such decay must occur by the Besicovitch projection theorem and by continuity of measures, since one takes the Lebesgue measure of decreasing sets in the parameter space of $\{\text{directions}\}\times\{\text{projected x values}\}$.}). Observe that $\G_n$ is in some sense comparable to an $L^{-n}$ neighborhood of $\G$(\footnote{an $L^{-n}$-neighborhood of $\G$ is contained in several small translates of $\G_n$, while $\G_n$ is contained in a neighborhood of size $\approx L^{-n}$ of $\G$}), so $\Fav(\G_n)$ is comparable to the likelihood that ``Buffon's needle" will land in a $L^{-n}$-neighborhood of $\G$.

The first quantitative results for the Favard length problem were obtained in \cite{PS},\cite{T}; in the latter paper a general way of making a quantitative statement from the Besicovitch theorem is considered. But being rather general, this method does not give a good estimate for self-similar structures such as  $\G_n$.

Indeed, vastly improved estimates have been proven in these cases: in \cite{NPV}, it was shown that for $1/4$ corner Cantor set one has $p<1/6$, such that $Fav(\K_n)\leq\frac{c_p}{n^{p}}$, and in \cite{BV0}, \cite{BV1} the same type power estimate was proved for the Sierpinski gasket $\Sier_n$ for some other $p>0$. These results cannot possibly be improved to $p=1$: $Fav(\K_n)\geq c\frac{\log\,n}{n}$.  (This is \cite{BV}(\footnote{the method is stable under ``bending the needle" slightly - see \cite{BV2}.}), and the argument and result also apply to $\Sier_n$.) Compare this with \cite{PS}, in which it was shown that certain random sets of which $\K_n$ is a special case almost surely decay in Favard length like $\frac1{n}$ in the liminf.

Crucial to \cite{BV} was a tiling property: namely, under orthogonal projection on the line with slope $1/2$, the squares composing $\K_n$ tile a line segment. Oddly enough, such a property can be used to prove upper bounds as well: under the assumption that some orthogonal projection in some direction contains an interval, Laba and Zhai \cite{LZ} showed that the result of \cite{NPV} holds for Cantor-like product sets of finite $H^1$ measure (but with a smaller exponent). Their argument uses tiling results obtained in Kenyon \cite{kenyon} and Lagarias-Wang \cite{lawang} to fill in a gap where \cite{NPV} fails to generalize (more on this shortly).

With the exception of \cite{PS} and \cite{T}, the above papers all extract their results from information about $L^2$ norms of the projection multiplicity function, which counts how many squares (or discs) project to cover each point. The function $f_{n,\theta}:\R\to\mathbb N$ is defined by 
$$
f_{n,\theta}=\sum_{\text{discs T of }\G_n}{\chi_{proj_\theta(T)}}.
$$ 

Note that $Fav(\G_n)=\pi^{-1}\int_0^{\pi}|\supp(f_{n,\theta})|d\theta$. In \cite{NPV} and \cite{BV0}, the $L^2$ norm of the analog of this function for squares was studied to obtain Buffon needle probability estimates for $\K_n$ -- in \cite{BV0}, $p=1,2$, were related to $\chi_{\supp(f_{n,\theta})}$ via the Cauchy inequality, while in \cite{NPV}, $p=2$ was studied via Fourier transforms and related to the measure of the level sets of $f_{n,\theta}=f$.

Consider some heuristics. Forgetting for the moment about angles, let $f:[0,1]\to\mathbb{N}$ be any sum of measurable characteristic functions such that $||f||_{L^1}=1$. If the mass is concentrated on a small set, the $L^p$ norm should be large for $p>1$. Thus a large $L^p$ norm should indicate that the support of a function is small, and vice versa. Let $K>0$, let $A=\supp\{f\}$, and let $A_K =\{x:f\geq K\}$. $1=\int f\leq ||f||_p||\chi_A||_q$, so $m(A)\geq ||f||_p^{-q}$, a decent estimate. The other basic estimate is not so sharp: $m(A)\leq 1-(K-1)m(A_K)$. However, a combinatorial self-similarity argument of \cite{NPV} shows that for the Favard length problem, it bootstraps well under further iterations of the similarity maps - this argument is revisited in Section $\ref{combinatorics}$. Hence, up to some loss of sharpness, it has been shown that to study Favard length of these self-similar sets, it is sufficient to study the $L^2$ norms of $f_{n,\theta}$.(\footnote{So far, only $L^p$ for $p=1,2,$ or $\infty$ have played any useful role, to our knowledge.})

One must average $|\supp{f_{n,\theta}}|$ over the parameter $\theta$ to get Favard length of $\G_n$. For $\K_n$ and $\Sier_n$, there are some directions for which the orthogonal projections do not even decay to length zero with $n$ (i.e., the $L^2$ norms of $f_{n,\theta}$ are bounded for these angles), and this countable dense set of directions is to a large extent classified in \cite{kenyon}. In \cite{NPV}, a method for controlling the measure of a set of angles $E$ on which the projections fail to decay rapidly was found: one takes the Fourier transform of $f_{n,\theta}$ in the length variable, and takes a sample integral of $|\hat{f}_{n,\theta}(x)|^2$ over a chosen small interval $I$ where $\int_{E\times I}|\hat{f}_{n,\theta}(x)|^2d\theta dx$ is small. One then shows that there is a $\theta\in E$ such that $\int_I|\hat{f}_{n,\theta}(x)|^2dx$ is large relative to $|E|$, and so $|E|$ must be small.

In all cases, $\hat{f}_{n,\theta}$ is a decay factor times a self-similar product $\prod_k\varphi_\theta(L^{-k}y)$ of trigonometric polynomials $\varphi_\theta$. The danger is that the low-frequency zeroes might kill off the better-behaved high-frequency terms. In \cite{NPV}, the four frequencies of $\varphi_\theta$ were symmetric around 0, allowing the terms to simplify to two cosines, and trigonometric identities allowed the whole product to be estimated by a single sine term. In \cite{LZ}, an analogous role was played by tilings of the line on the non-Fourier side by $proj_{\theta_0}(\G_n)$ in the special direction $\theta_0$, and the product structure of $\G_n$ allowed for a change and separation of variables.

Separating variables is more difficult when there is no product structure. The simplest case without the product structure is the Sierpinski gasket $\Sier$ considered in Section \ref{Sierp}. We give there a sketch of the power estimate (proven in detail in \cite{BV0}), which is based on the fact that zeroes of $\varphi(3^k\cdot)$ are separated away from each other for different values of $k$. This special structure of zeros (we call it ``analytic tiling" after \cite{LZ}) is not always available for all angles. We have not yet found an adequate substitute for it in the general case, and this is why the for the general case we still only have $Fav(\G_n)\leq e^{-\epsilon_0\sqrt{\log\, n}}$.

Rather strangely, a claim in the spirit of the Carleson Embedding Theorem, in the form of Lemma \ref{CETSQ}, plays an important part in our reasoning. Because the Fourier transform turns stacks of discs (i.e., sums of overlapping characteristic functions) into clusters of frequencies, this lemma provides important upper bounds when $\theta$ belongs to $E$.

The main result of this article is the following estimate.
\begin{theorem}
\label{mainth1}For all $n\in\mathbb{N}$,
$$
\Fav(\G_n)\le C\, e^{-\epsilon_0\sqrt{\log n}}\,.
$$
\end{theorem}

For the Sierpinski gasket, the result is exactly that of \cite{NPV}:
\begin{theorem}
\label{mainthsier}For all $p<1/12$, there exists $C_p>0$ such that for all $n\in\mathbb{N}$,
$$
\Fav(\Sier_n)\leq C_pn^{-p}.
$$
\end{theorem}

\bigskip

\noindent{\bf Acknowledgements}. We are greatly indebted to  Fedja Nazarov  for many valuable discussions, we also express our deep gratitude to Izabella Laba for many useful conversations and to John Garnett who introduced the subject to one of the authors.

\bigskip

\section{Definitions and result for Sierpinski gasket. Sketch of the proof for Sierpinski gasket}
\label{Sierp}

We give in the first section the sketch of the estimate for a special self-similar set--the Sierpinski gasket. The structure of zeros of a certain trigonometric polynomial related to this set plays the crucial part in this ``good'' (meaning power) estimate of the type $\le n^{-c}$. We put this sketch here for the reader to be able to compare it with the general case, where we get only the estimate $\le e^{-c\sqrt{\log n}}$ due to the lack of understanding of these zeros. It is elaborated in more detail in \cite{BV0}.
It may be instructive to compare the general and the special cases. If the reader is interested only in the general case then he/she can skip the present section and go directly to the next one. Conversely, the reduced detail and difficulty may benefit the first-time reader looking for the general overview of the method.

\bigskip

$B(z_0,r):=\{z\in\C:|z-z_0|<r\}$. For $\alpha\in\lbrace -1, 0, 1\rbrace^n$ let $$z_\alpha:=\sum_{k=1}^n{(\frac{1}{3})^ke^{i\pi[\frac{1}{2}+\frac{2}{3}\alpha_k]}},\,\,\,\Sier_n:=\bigcup_{\alpha\in\lbrace -1, 0, 1\rbrace^n} B(z_\alpha,3^{-n}).$$
This set is our approximation of a partial Sierpinski gasket; it is strictly larger. We may still speak of the approximating discs as ``Sierpinski triangles."

The result for the Sierpinski gasket is the following:
\begin{theorem}\label{main}
For some $c>0$,
$\Fav(\Sier_n)\leq\frac{C}{n^{c}}.$
\end{theorem}
We will simplify the proof by picking specific values for constants; at the end of this paper, a short remark shows how to recover the full range $c<1/6$ as in Theorem $\ref{mainthsier}$.

The set $\Sier_n$ is $3^{-n}$ approximation to a Besicovitch irregular set (see \cite{falc1} for definition) called the Sierpinski gasket. Recently one detects a considerable interest in estimating the Favard length of such $\epsilon$-neighborhoods of Besicovitch irregular sets, see \cite{PS}, \cite{T}, \cite{NPV}, \cite{LZ}. In \cite{PS} a random model of such Cantor set is considered and estimate $\asymp \frac1n$ is proved. But for non-random self-similar  sets the estimates of \cite{PS} are more in terms  of $\frac1{\log\cdots\log n}$ (number of logarithms depending on $n$) and more suitable for general class of ``quantitatively Besicovitch irregular  sets" treated in \cite{T}.

As in the introduction, let
$$f_{n,\theta}:=\sum_{\text{Discs D of }\Sier_n}\chi_{proj_\theta(D)}.$$
Self-similarity allows us to write $f_{n,\theta}$ in a form well-suited to Fourier analysis:
$f_{n,\theta}=\frac1{2}\nu_n *3^n\chi_{[-3^{-n},3^{-n}]},$ where
$$\nu_n :=*_{k=1}^n\widetilde{\nu}_k\text{ and }\widetilde{\nu}_k:=\frac{1}{3}[\delta_{3^{-k}cos(\pi/2 -\theta)} +\delta_{3^{-k}cos(-\pi/6 -\theta)} +\delta_{3^{-k}cos(7\pi/6 -\theta)}].$$

For $K>0$, let $A_K:= A_{K,n,\theta}:=\{x:f_{n,\theta}\geq K\}$. Let $\LL_{\theta,n}:=\text{proj}_\theta(\G_n)=A_{1,n,\theta}$.
For our result, some maximal versions of these are needed:
$$f_{N,\theta}^*:=max_{n\leq N}f_{n,\theta},\,\,\, A_K^*:= A_{K,n,\theta}^*:=\{x:f_{n,\theta}^*\geq K\}.$$

Also, let $E:=E_N:=\{\theta:|A_K^*|\leq K^{-3}\}$ for $K=N^{\epsilon_0}$, where $\epsilon_0>0$ is a small enough absolute constant.

Later, we will jump to the Fourier side, where the function $$\varphi_\theta(x):=\frac1{3}[e^{-i\cos(\pi/2-\theta)x}+e^{-i\cos(-\pi/6-\theta)x}+e^{-i\cos(7\pi/6-\theta)x}]$$
plays the central role: $\widehat{\nu_n}(x)=\prod_{k=1}^n\varphi_\theta(3^{-k}x)$.

\subsection{General philosophy}
\label{phil}

Fix $\theta$. If the mass of $f_{n,\theta}$ is concentrated on a small set, then $||f_{n,\theta}||_p$ should be large for $p>1$ - and vice versa. $1=\int f\leq ||f_{n,\theta}||_p||\chi_{\LL_{\theta,n}}||_q$, so $m(\LL_{\theta,n})\geq ||f||_p^{-q}$, a decent estimate. The other basic estimate is not so sharp: 
\begin{equation}\label{naive}m(\LL_{\theta,N})\leq 1-(K-1)m(A_{K,N,\theta})\end{equation}
However, a combinatorial self-similarity argument of \cite{NPV} and revisited in \cite{BV} shows that for the Favard length problem, it bootstraps well under further iterations of the similarity maps:

\begin{theorem}
\label{}
If $\theta\notin E_N$, then $|\LL_{\theta,NK^3}|\leq\frac{C}{K}$.
\end{theorem}

This is proved in full detail as Theorem $\ref{gooddir1}$. Note that the maximal version $f_N^*$ is used here. A stack of $K$ triangles at stage $n$ generally accounts for more stacking per step the smaller $n$ is. Thus the maximal function $f_N^*$ captures this information by recording a large level set at height $K$ whenever $f_n$ attains height $K$ for a small value of $n$. For fixed $x\in A_{K,N,\theta}^*$, the above theorem considers the smallest $n$ such that $x\in A_{K,n,\theta}$, and uses self-similarity and the Hardy-Littlewood theorem to prove its claim by successively refining an estimate in the spirit of \eqref{naive}.
Of course, now Theorem \ref{main} follows from the following:
\begin{theorem}\label{mofE}
Let $\epsilon_0<1/\log_3(169)$, sufficiently, $\epsilon_0\leq 1/9.262$. Then for $N>>1$, $|E|<N^{-\epsilon_0}$.
\end{theorem}
It turns out that $L^2$ theory on the Fourier side is of great use here. The following is later proved as Theorem $\ref{combin}$:
\begin{theorem}\label{L2}
For all $\theta\in E_N$ and for all $n\leq N$, $||f_{n,\theta}||_{L^2}^2\leq CK$.
\end{theorem}
One can then take small sample integrals on the Fourier side and look for lower bounds as well.
Let $K=N^{\epsilon_0}$, and let $m=2\epsilon_0 \log_3N$. Theorem \ref{L2} easily implies the existence of $\tilde{E}\subset E$ such that $|\tilde{E}|>|E|/2$ and number $n$, $N/4<n<N/2$, such that for all $\theta\in\tilde{E}$,
$$\int_{3^{n-m}}^{3^n}{\prod_{k=0}^n{|\varphi_\theta(3^{-k}x)|^2}dx}\leq \frac{2CKm}{N}\leq 2\eps_0N^{\eps_0-1}\log N.$$
Number $n$ does not depend on $\theta$; $n$ can be chosen to satisfy the estimate in the average over $\theta\in E$, and then one chooses $\tilde{E}$. Let $I:=[3^{n-m},3^n].$

Now the main result amounts to this (with absolute constant $\alpha$ large enough):
\begin{theorem}\label{lower}
$$\exists\theta\in\tilde{E}: \int_I{\prod_{k=0}^n{|\varphi_\theta(3^{-k}x)|^2}dx}\geq c3^{m-2\cdot\alpha m}=cN^{-2\epsilon_0(2\alpha-1)}.$$
\end{theorem}
The result: $2\eps_0 \log N\geq N^{1-\eps_0(4\alpha-1)}$, i.e., $N\leq N^*$.
Now we sketch the proof of Theorem \ref{lower}. We split up the product into two parts: high and low-frequency: $P_{1,\theta}(z)=\prod_{k=0}^{n-m-1}\varphi_{\theta}(3^{-k}z)$, $ P_{2,\theta}(z)=\prod_{k=n-m}^n \varphi_{\theta}(3^{-k}z)$.

\begin{theorem}\label{P1}
For all $\theta\in E$, $\int_I|P_{1,\theta}|^2\,dx\geq C\,3^m\,.$
\end{theorem}
 Low frequency terms do not have as much regularity, so we must control the damage caused by the \textbf{set of small values}, $SSV(\theta):=\{x\in I:|P_2(x)|\leq 3^{-\ell}\}$, $\ell=\alpha\,m$. In the next result we claim the existence of $\mathcal{E}\subset\tilde{E}$, $|\mathcal{E}|>|\tilde{E}|/2$ with the following property:

\begin{theorem}\label{P1SSV}
$$\int_{\tilde{E}}\int_{SSV(\theta)}|P_{1,\theta}(x)|^2dx\,d\theta\leq 3^{2m-\ell/2}\Rightarrow \forall \theta\in \mathcal{E}\,\,\int_{SSV(\theta)}|P_{1,\theta}(x)|^2dx\leq c\,K\,3^{2m-\ell/2}\,.$$
\end{theorem}
 Then Theorems \ref{P1} and \ref{P1SSV} give Theorem \ref{lower}; since $\ell=\alpha m$ and $K^2=3^m$, we see that any $\alpha>2$ may be used for this estimate; however, we will need $\alpha$ to be larger soon.

\subsection{Locating zeros of $P_2$}
\label{zP2}
We can consider $\Phi(x, y) = 1+ e^{ix} +e^{iy}$. The key observations (see also the discussion section at the end of the paper) are
\begin{equation}\label{keyobs}
|\Phi(x,y)|^2 \ge a(|4\cos^2\,x -1|^2 +|4\cos^2 \,y-1|^2)\,,\,\,\,\frac{\sin 3x}{\sin x} = 4\cos^2 \,x -1\,.
\end{equation}
Actually, we will set $\alpha=a^{-1}$ in the end. Changing variable we can replace $3\varphi_{\theta} (x)$ by $\phi_t(x)=\Phi (x, tx)$.

Consider $P_{2,t}(x):=\prod_{k=n-m}^{n} \frac13\phi_t(3^{-k} x)$, $P_{1,t}(x):=\prod_{k=0}^{n-m} \frac13\phi_t(3^{-k} x)$.

We need $SSV(t):=\{x\in I: |P_{2,t}(x)|\le 3^{-\ell}\}$. One can easily imagine it if one considers $\Omega:=\{(x,y)\in [0,2\pi]^2: |\mathcal{P}(x,y)|:=|\prod_{k=0}^m \Phi( 3^k x, 3^k y) |\le 3^{m-\ell}\}$. Moreover, (using that if $x\in SSV(t)$ then $3^{-n}x\ge 3^{-m}$, and using $xdxdt=dxdy$) we change variable in the next integral:
$$
\int_{\tilde{E}}\int_{SSV(t)} |P_{1,t}(x)|^2\,dxdt =3^{-2n+2m}\cdot 3^n\int_{\tilde{E}}\int_{3^{-n}SSV(t)}|\prod_{k=m}^n \Phi( 3^k x, 3^k tx)|^2\,dxdt \le 
$$
$$
3^{-n+3m}\int_{\Omega} |\prod_{k=m}^n \Phi( 3^k x, 3^k y)|^2\,dxdy\,.
$$
Now notice that by our key observations 
\begin{equation}\label{omegain}
\Omega\subset \{(x,y)\in [0,2\pi]^2: |\sin 3^{m+1} x|^2 +|\sin 3^{m+1}y|^2 \le a^{-m} 3^{2m-2\ell}\le  3^{-\ell}\}\,.
\end{equation}
The latter set $\mathcal{Q}$ is the union of $4\cdot 3^{2m+2}$ squares $Q$ of size $3^{-m-\ell/2}\times 3^{-m-\ell/2}$. Fix such a $Q$ and estimate
$$
\int_{Q} |\prod_{k=m}^n \Phi( 3^k x, 3^k y)|^2\,dxdy\le 3^{\ell} \int_{Q} |\prod_{k=m+\ell/2}^n \Phi( 3^k x, 3^k y)|^2\,dxdy \le 
$$
$$
3^{\ell} \cdot (3^{-m-\ell/2})^2\int_{[0,2\pi]^2}|\prod_{k=0}^{n-m-\ell/2} \Phi( 3^k x, 3^k y)|^2\,dxdy \le
3^{\ell} \cdot (3^{-m-\ell/2})^2\cdot 3^{n-m-\ell/2}= 3^{-2m} \cdot 3^{n-m-\ell/2}\,.
$$
Therefore, taking into account the number of squares $Q$ in $\mathcal{Q}$ and the previous estimates we get

$$
\int_E\int_{SSV(t)} |P_{1,t}(x)|^2\,dxdt\le 3^{2m-\ell/2}\,.
$$
Theorem \ref{P1SSV} is proved.

%
%
\bigskip

\noindent{\bf Remarks}. The rest of the paper is devoted to general self-similar sets, where we can get only a $e^{-c\sqrt{\log n}}$ result due to the lack of structure (possible lack of ``analytic tiling") of the zeros of trigonometric polynomials, which are ``telescopic products" of one trigonometric polynomial. See the last section of this work for the discussion.

It is true that $\alpha$ depends on the constant $a$ in $\eqref{keyobs}$, since it appears in $\eqref{omegain}$. One can use $a=\frac1{18}$, attained at $(x,y)=(0,\pi)$. Then from $\eqref{omegain}$, we get $\alpha=m/\ell\geq log_3(162)\approx 4.631$ as our last condition on $\alpha$. We need this to compute the best exponent $p$.

Note that in our argument, we cut a couple corners. To get the best exponent currently available, let $\gamma>1$. Let $m=\gamma\epsilon_0\log_3 N$. Then the argument works as long as $\epsilon_0<[2\gamma\alpha + 1 - \gamma]^{-1}$, i.e., $\epsilon_0<\frac1{2\log_3(169)}$. Using the sharper exponent $\beta>2$ in Theorem $\ref{gooddir1}$, one can get any $p=\frac1{\epsilon_0^{-1}+\beta}<\frac1{[2\log_3(169)]^{-1}+2}$ in the estimate $\Fav(\Sier_n)\leq C_p n^{-p}$. In particular, $p<\frac1{11.262}$ is small enough.

This can be improved if more care is taken, but not beyond $p=1/6$ without a substantial change in the overall approach.

\bigskip

\section{{The Fourier-analytic part}} 
\label{sec:fourier}
\subsection{{The setup}} \label{subsec:setup}
The goal of this section is to prove Theorem $\ref{mofE}$, which shows that for most directions, a considerable amount of stacking occurs when the discs are projected down. Throughout the paper, the constants $c$ and $C$ will vary from line to line, but will be absolute constants not depending on anything. The symbols $c$ and $C$ will typically denote constants that are sufficiently small or large, respectively. Everywhere we use the definition $B(z_0,\varepsilon):=\{z\in\C:|z-z_0|<\varepsilon\}$.

Let
$$\G_1:=\bigcup_{j=1}^L B(r_je^{i\theta_j},\frac1{L}).$$

Then one constructs $\G$ and $\G_n$ using the similarity maps of the unit disc onto the discs forming $\G_1$. For convenience, we will now rescale $\G_n$ by a factor absolutely comparable to 1 and bound the discs of $\G_n$ by slightly larger discs and study this set instead. 

Recall that

$$f_{n,\theta}:=\sum_{\text{Discs }D\text{ of }\G_n}^L \chi_{proj_\theta(D)},$$

Observe that $f_{n,\theta}=\nu_n *L^n\chi_{[-L^{-n},L^{-n}]},$ where $\nu_n :=*_{k=1}^n\widetilde{\nu}_k$ and 
$$
\widetilde{\nu}_k=\frac{1}{L}[\sum_{l=1}^L\delta_{L^{-k}r_l\cos(\theta-\theta_l)}].
$$

We will now slightly modify $f$ for convenience. Note that

$$
\hat{f}_{n,\theta}(x)=L^n\hat{\chi}_{[-L^{-n},L^{-n}]}(x)\cdot\prod_{k=1}^n\phi_\theta(L^{-k}x),
$$
where $\phi_\theta(x)=\frac1{L}[\sum_{l=1}^Le^{-ir_l\cos(\theta_l-\theta)x} ]$. By factoring and changing the variable, we may instead write in place of $\phi_\theta$ the function

\begin{equation}
\label{phit}
\varphi_t(x)=\frac1{L}[1+e^{ix}+e^{itx}+\sum_{l=4}^L e^{a_lx +b_ltx}],\,\,\, t\in [0,1]\,.
\end{equation}
We assumed here that $r_1=0$, $r_2=r_3=1$, $\theta_2=0$, $\theta_3=\pi/2$. We can do this by affine change of variable.

For numbers $K,N>0$, define the following:

\begin{equation}\label{fNstar}f_N^*(s):=f_{N,t}^*\sup_{n\le N} f_{n,t}(s)\end{equation}
\begin{equation}\label{AKstar}A^*_K:=A^*_{K,N,t}:=\{s:f_N^*(s)\geq K\}\end{equation}
\begin{equation}\label{Edef}E:=\lbrace t : |A^*_K|\leq \frac1{K^3}\rbrace\,.\end{equation}

$E$ is essentially the set of pathological $t$ such that $||f_{n,t}||_{L^2(s)}$ is small for all $n\leq N$, as in \cite{NPV}. In fact, we have this result, proved in Section $\ref{combi}$:
 
 \begin{theorem}
 \label{combin1}
 Let $t\in E$. Then
 $$
 \max_{0\le n\le N} \|f_{n,t}\|^2_{L^2(s)}\le c\, K\,.
 $$
 \end{theorem}

The aim of Section $\ref{sec:fourier}$ is to prove the following:

 \begin{theorem}
 \label{mofE}
Let $\eps_0$ be a fixed small enough constant. Then for $N>>1$, $|E|<e^{-\eps_0(\log N)^{1/2}}$.
 \end{theorem}
 
So let $K\approx e^{\eps_0(\log N)^{1/2}}$, and suppose $|E|>\frac1{K}$. We will show that $N<N^*$, for some finite constant $N^*>>1$.

\subsection{{\bf Initial reductions}}\label{subsec:reductions}

Because of Theorem $\ref{combin1}$, we have $\forall t\in E$,
\begin{equation}
\label{intnuhat}
K\geq  ||f_{N,t}||^2_{L^2(s)}\approx ||\widehat{f_{N,t}}||^2_{L^2(x)}\geq C\int_1^{L^{N/2}}{|\widehat{\nu_N}(x)|^2dx}
\end{equation}

Let $m\approx (\frac{\eps_0}{2}\log N)^{1/2},$ $K\approx\log N$. Split $[1,L^{N/2}]$ into $N/2$ pieces $[L^k,L^{k+1}]$ and take a sample integral of $|\widehat{\nu_N}|^2$ on a small block $I:= [L^{n-m},L^n]$, with $n\in [N/4,N/2]$ chosen so that
$$
\frac{1}{|E|}\int_E{\int_{L^{n-m}}^{L^n}{|\widehat{\nu_N}(x)|^2dx\,dt}}\leq CKm/N\,.
$$
This choice is possible by $\eqref{intnuhat}$. Define 
 $$
 \tilde{E}:=\lbrace t\in E: \int_{L^{n-m}}^{L^n}{|\widehat{\nu_N}(x)|^2dx}\leq 2CKm/N\rbrace\,.
 $$
  It then follows that $|\tilde{E}|\geq\frac{1}{2K}$.
  

Note that $\widehat{\nu_N}(x)=\prod_{k=1}^N{\varphi (L^{-k}x)}\approx\prod_{k=1}^n{\varphi (L^{-k}x)}$ for $x\in [L^{n-m},L^n]$.

So for $t\in E$,
$$\int_{L^{n-m}}^{L^n}{\prod_{k=1}^n{|\varphi_t(L^{-k}x)|^2}dx}\leq \frac{CKm}{N}\leq 2\eps_0N^{\eps_0-1}\log N.$$
Recall that $m\approx (\frac{\eps_0}{2}\log N)^{1/2}$. Later, we will show that $\exists t\in E$ and absolute constant $\alpha$ such that 

\begin{equation}\label{totalest}
\int_{L^{n-m}}^{L^n}{\prod_{k=1}^n{|\varphi_t(L^{-k}x)|^2}dx}\geq cL^{m-2\cdot\alpha m^2}\geq cN^{-\alpha\epsilon_0}.
\end{equation}

The result: $2\eps_0 \log N\geq N^{1-4\alpha\eps_0-\eps_0}$, i.e., $N\leq N^*$ if $\epsilon_0$ is small enough. In other words:

\begin{prop}
\label{reduction}
Inequality $\eqref{totalest}$ is sufficient to prove Theorem $\eqref{mofE}$. Further, inequality $\ref{totalest}$ can be deduced from Propositions $\ref{P1below}$ and $\ref{P1onSSV}$, as will be seen shortly.
\end{prop}

So let us prove inequality $\eqref{totalest}$.

First, let us write $\prod_{k=1}^n\varphi_t(L^{-k}x)=P_t(x)=P_{1,t}(x)P_{2,t}(y)$, where $P_2$ is the low frequency part, and $P_1$ is has medium and high frequencies:
$$P_{1,t}(x):=\prod^{n-m}_{k=1}\varphi_t(L^{-k}x)=\widehat{\nu_{n-m}}(x)$$ $$P_{2,t}(x)=\prod_{k=n-m}^{n}\varphi_t(L^{-k}x)=\widehat{\nu_m}(L^{m-n}x)$$
We want the following:

\begin{prop}\label{P1below}
Let $t\in E$ be fixed. Then $\int_{L^{n-m}}^{L^n}{|P_{1,t}(x)|^2dx}\geq C\,L^m$.
\end{prop}

Recall that we defined the set $\tilde{E}, |\tilde{E}|> |E|/2$, and we assume that
\begin{equation}
\label{cannothide}
|E|>1/K\,.
\end{equation}
Recall that we denoted
$$
I= [L^{n-m}, L^n]\,.
$$
We also want a proportion of the contribution to the integral separated away from the complex zeroes of $P_{2,t}$:

\begin{prop}\label{P1onSSV}
Let $SSV(t):=\{x\in I: |P_{2,t}(x)|\leq L^{-\alpha m^2}\}$. Suppose also that $E$ is unable to hide, that is \eqref{cannothide} is valid. Then there  exists a subset $\mathcal{E}\subset \tilde{E},\,\,|\mathcal{E}|\geq 1/4K,$ such that for every $\theta \in \mathcal{E}$ one has
$$
\int_{SSV(t)}|P_{1,t}(x)|^2dxdt\leq 2c\,L^m\,,
$$
where $2c$ is less than the $C$ from Proposition $\ref{P1below}$.
In particular, 
$$
\frac1{|\tilde{E}|}\int_{\tilde{E}}\int_{SSV(t)}|P_{1,t}(x)|^2dxdt\leq c\,L^m,
$$
\end{prop}
\noindent{\bf Remarks.} 1) The set $SSV(t)$ is so named because it is the $\textbf{set of small values}$ of $P_2$ on $I$.  Combining this with Proposition $\ref{P1below},$
$$
\int_{L^{n-m}}^{L^n}|P_{1,t}(x)|^2|P_{2,t}(x)|^2\, dx \ge\int_{I\setminus SSV(t)}|P_{1,t}(x)|^2\cdot L^{-\alpha m^2}\,dx\geq c\,L^{m-2\alpha m^2},
$$
which gives \eqref{totalest}--exactly what we promised to obtain from Propositions \ref{P1onSSV}, \ref{P1below}.

\noindent 2) 
Thus Propositions $\ref{P1below}$ and $\ref{P1onSSV}$ suffice to prove Theorem $\ref{mofE}$, and Proposition $\ref{reduction}$ has been demonstrated.

\noindent 3)  All this holds if \eqref{cannothide} holds. But if we have the opposite:
\begin{equation}
\label{Esmall}
|E| \le 1/K = L^{-\frac{m}2} =e^{-C(L)\epsilon_0(\log N)^{1/2}}\,,
\end{equation}
the main result is automatically proved because we have only a small set of singular directions.

First, let us fix $t\in E$ and prove Proposition $\ref{P1below}$.

\begin{proof}
We are using  first Salem's trick on $$\int_0^{L^n}{|P_1(x)|^2dx}:$$

Let $h(x):=(1-|x|)\chi_{[-1,1]}(x)$, and note that $\hat{h}(\alpha)=C\frac{1-cos\alpha}{\alpha ^2}>0$. Then if we write $P_1=L^{m-n-1}\sum_{j=0}^{L^{n-m}}{e^{i\alpha_jx}}$, we get
$$\int_0^{L^n}{|P_1(x)|^2dx}\geq 2\int_{-L^{n}}^{L^{n}}{h(L^{-n}x)|P_1(x)|^2dx}$$
$$\geq C(L^{m-n})^2[L^n\cdot L^{n-m}+\sum_{j\neq k; j,k=1}^{L^{n-m}}L^n{\hat{h}(L^n(\alpha_j-\alpha_k))}]\geq CL^m.$$

To show that this is not concentrated on $[0,L^{n-m}]$, we will use Theorem \ref{combin1} and Lemma \ref{CETSQ}. We get
$$
\int_0^{L^{n-m}}{|P_1(x)|^2dx}= \int_0^{L^{n-m}}{|\widehat{\nu_{n-m}}(x)|^2dx}=L^{2(m-n)}\int_0^{L^{n-m}}|\sum_{j=0}^{n-m}e^{i\alpha_jx}|^2dx$$
$$\leq CK\leq CL^{\frac{m}{2}}.
$$

\end{proof}

So now we have Proposition $\ref{P1below}$. The greater challenge will be Proposition $\ref{P1onSSV}$.

\subsection{The proof of Proposition $\ref{P1onSSV}$}
\label{subsec:P1onSSV}

Recall that $SSV(t):=\{ x\in I=[L^{n-m}, L^n]: |P_{2,t}(x)|\leq {L^{-\alpha m^2}} \}$.

To get Proposition $\ref{P1onSSV}$, we will split $P_{1,t}$ into two parts, $\Pshp$ and $\Pflt$ corresponding to medium and high frequencies.

 A straightforward application of Lemma $\ref{CETSQ}$ to high frequency  part $\Pshp$ will get us part of the way there, see Proposition \ref{Pshpest}  (for fixed $t$, the size of $SSV(t)$ does not overwhelm the average smallness of $\Pshp$), and the claim \ref{numbertheory}  applied to medium frequency term  $\Pflt$ will further sharpen the final estimate to what we need.

Naturally, $\Pflt$ and $\Pshp$ are defined as the medium and high frequency parts of $P_{1,t}(x)$. Below, $\ell:= \alpha m$:
$$
\Pflt := \prod_{k=n-m-\ell}^{n-m-1}\varphi_t(L^{-k}x)=\widehat{\nu_{\ell-1}}(L^{m+\ell-n}x)\,,\,
$$
$$
\Pshp:= \prod_{k=1}^{n-m-\ell-1}\varphi_t(L^{-k}x)=\hat{\nu}_{n-m-\ell-1}(x).
$$

Here is the first claim of this subsection

\begin{prop}
\label{numbertheory}
For all sufficiently small positive numbers $\tau\le \tau_0$ and for all sufficiently large $m$ and $\ell=\alpha\,m$ there exists an exceptional set  $H$ of directions $t$ such that
\begin{equation}
\label{H1}
|H| \le L^{-\ell/2}\,,
\end{equation}
\begin{equation}
\label{H2}
\forall t\notin H\, \forall x\in [L^{n-m}, L^n],\, \, |\Pflt| \le e^{-\tau\,\ell}\,.
\end{equation}
\end{prop}

\begin{proof}

Notice that
$$
\phi_{\theta}(r) =\Phi (r\cos\theta, r\sin\theta)\,,
$$
where for $x=(x_1,x_2)$,
$$
\Phi(x) :=\Phi(x_1, x_2) = \frac1L\sum_{l=1}^L  e^{2\pi i \langle a_l,x\rangle}\,.
$$
As some pair of vectors $a_l-a_1$, $l\in [1,L]$ must span a two-dimensional space, we can assume without the loss of generality (make an affine change of variable) that
$$
a_1=(0,0)\,, a_2=(1,0)\,, a_3=(0,1)\,.
$$
Then
\begin{equation}
\label{Phiview}
\Phi(x_1, x_2) = \frac1L(1+e^{2\pi i x_1} + e^{2\pi i x_2} + \sum_{l=4}^L  e^{2\pi i \langle a_l,x\rangle})\,.
\end{equation}
We make the change of variable
$y=(y_1, y_2) = L^{-(n-m)} x$.   Let $R_t$ denote the ray $y_2=ty_1$. Then we need to prove that there exists a small set $H$ of $t'$s such that if $y\in R_t \cap \{y:|y|\in [1, L^m]\}$, $t\notin H$ then
\begin{equation}
\label{PhiH}
|\Phi(y)\cdot\dots\cdot\Phi(L^{\ell}y)| \le e^{-\tau\ell}\,.
\end{equation}

We consider only the case $t\in [0,1]$, all our $y$'s will be such that $0<y_2\le y_1$, and as $|y|\ge 1$ we have $y_1\ge \frac1{\sqrt{2}}$.

It is very difficult if at all possible for function $\Phi$ to satisfy $|\Phi(y) |=1$. In fact, looking at \eqref{Phiview} we can see that
\begin{equation}
\label{dist}
|\Phi(y)| \le 1-b\dist(y, \Z^2) \le e^{-b\dist(y, \Z^2)}\,.
\end{equation}

Therefore, we are left to understand that there are few $t$'s such that
\begin{equation}
\label{Hopposite}
\exists y\in R_t, : y_1\in [ \frac1{\sqrt{2}}, \,L^m] \,:\,\,\,b\cdot\sum_{k=0}^{\ell} \dist (L^k\,y, \Z^2) \le \tau\,\ell\,.
\end{equation}

Fix $y\in R_t$  as above. If \eqref{Hopposite} holds then for $90$ per cent of $k's$ one has

\begin{equation}
\label{Hopposite1}
 \dist (L^k\,y, \Z^2) \le 10\tau\,\ell\,.
\end{equation}

Denote $Z_y:=\{ k\in [0,\ell]:  \dist (L^k\,y, \Z^2) \le 10\tau\,\ell\}$. We know that
$$
|Z_y|\ge 0.9 \ell\,.
$$
 Let us call {\bf scenario} the collection $s:=\{ m_1; k_1,..., k_{0.1\ell}\}$, where $m_1=0,.., m;  0\le k_1<...<k_{0.1\ell}$. 
 
 Every $t$ such that there exists $y$ such that \eqref{Hopposite} holds generates  several scenarios according to 
 $$
 y_1\in [L^{m_1-1}, L^{m_1})
 $$
 and according to what is the set $[0,\ell]\setminus Z_y$---this is the set  $k_1,..., k_{0.1\ell}$ of the scenario. 
 
 We will calculate the number of scenarios later. Now let us fix a scenario $s=\{ m_1; k_1,..., k_{0.1\ell}\}$, and let us estimate the measure of the set  $T(s)$,
 $T(s):=\{t\in (0,1): \exists y, y_2=ty_1, \,y_1\in [L^{m_1-1}, L^{m_1})\,\, \text{such that}\,\,[0,\ell]\setminus Z_y=\{k_1,\dots, k_{0.1\ell}\}$. To do that for this fixed scenario we fix a {\bf net}. 
 To explain what is a net we fix
 $$
 a:= \bigg[\frac{\log\frac{100}{\eta}}{\log L}\bigg] +1\,,
 $$
 where $\eta= C\, \tau$ and $C$ is an absolute constant to be chosen soon.
 
 A net is a collection $N(s):=\{ n_1,\dots, n_{j}\}$, $n_1<n_2<\dots$, where every  $n_i$ is not among $k_j$ included in the scenario,    $j\ge  \frac{3\ell}{4a}+1$, and
 $$
 n_{i+1}-n_i \ge 2a\,.
 $$
 Given a scenario it is always possible to built a net. In fact we just delete from $[0,\ell]$ the numbers $k_1,..., k_{0.1\ell}$ belonging to the scenario, we are left with at least $0.9\ell$ numbers.
 We choose an arithmetic progression with step $a$ (enumerating them anew first).  This arithmetic progression will be long enough, its length $j\ge  \frac{3\ell}{4a}$ because after eliminating $k_1,..., k_{0.1\ell}$ we still have at least $0.9\ell$ numbers left. We mark the numbers of this progression. Then we put back $k_1,..., k_{0.1\ell}$.  The marked numbers will form our net.
 
 If  $t\in T(s)$ then there exists $y=(y_1, ty_1)$ as above, in particular,
 $$
  \dist (L^{n_i}\,y, \Z^2) \le 10\tau\,\ell\,,\,\, \forall n_i\in N(s)\,.
  $$
  Let us write that then there exist integers $p_1\le q_1$: $|L^{n_1} y_1- q_1| <10\tau$, $|L^{n_1} y_2- p_1| <10\tau$,  so
  $$
  \bigg|t-\frac{p_1}{q_1}\bigg| = \bigg|\frac{L^{n_1}y_2}{L^{n_1}y_1}-\frac{p_1}{q_1}\bigg| = \bigg|\frac{L^{n_1}y_2-p_1+p_1}{L^{n_1}y_1-q_1+q_1}-\frac{p_1}{q_1}\bigg|
  $$
  $$
  \bigg|\frac{(L^{n_1}y_2-p_1+p_1)q_1-(L^{n_1}y_1-q_1+q_1)p_1}{(L^{n_1} y_1-q_1+q_1) q_1}\bigg|\le \frac{|L^{n_1}y_2-p_1||q_1|+|L^{n_1}y_1-q_1||p_1|}{q_1-10\tau) q_1}\le 40\tau\frac1{q_1}\,.
  $$
  As promised we choose $C$: $C=40$, $\eta:=40\tau$ and we get
  \begin{equation}
  \label{q1}
  \exists p_1\le q_1\,:\,\,\bigg|t-\frac{p_1}{q_1}\bigg| \le \eta\frac1{q_1}\,.
  \end{equation}
  Next we choose integers $p_2\le q_2$: $|L^{n_2} y_1- q_2| <10\tau$, $|L^{n_2} y_2- p_2| <10\tau$ and obtain
 
\begin{equation}
  \label{q2}
   \exists p_2\le q_2\,:\,\,\bigg|t-\frac{p_2}{q_2}\bigg| \le \eta\frac1{q_2}\,.
  \end{equation}
  
  Notice also that because of $|L^{n_1} y_1- q_1| <10\eta$, $|L^{n_2} y_1- q_2| <10\eta$, $y_1\ge 1/\sqrt{2}$, and smallness of $\tau$, and the fact that $n_2-n_1\ge  2a$, we get
  \begin{equation}
  \label{q1q2}
  \frac{q_2}{q_1} \ge L^a \ge \frac{100}{\eta}\,.
  \end{equation}

We continue in the same vein, $i=2,\dots,j-1\ge \frac{3\ell}{4a}$:
\begin{equation}
  \label{qi}
 \exists p_i\le q_i\,:\,\, \bigg|t-\frac{p_i}{q_i}\bigg| \le \eta\frac1{q_i}\,.
  \end{equation}
  
  Notice also that because of $|L^{n_1} y_1- q_1| <10\eta$, $|L^{n_2} y_1- q_2| <10\eta$, $y_1\ge 1/\sqrt{2}$, and smallness of $\tau$, and the fact that $n_2-n_1\ge  2a$, we get
  \begin{equation}
  \label{qiq}
  \frac{q_{i+1}}{q_i} \ge L^a\ge \frac{100}{\eta}\,.
  \end{equation}

Inequality \eqref{q1} gives that $|T(s)| \le \eta$, inequalities \eqref{q1} and \eqref{q2} in conjunction with \eqref{q1q2} give $|T(s)|\le  \bigg(1+\frac1{100}\bigg)\eta^2$,
similarly all inequalities  \eqref{qi}, \eqref{qiq} together give
$$
|T(s)|\le (1.01\eta)^{\frac{3\ell}{4a}}\ge e^{0.1\,\ell} L^{-\frac34\ell(1-\epsilon(\eta))}\,.
$$
Here we used of course that $a:= \bigg[\frac{\log\frac{100}{\eta}}{\log L}\bigg] +1$.
Finally, if $\eta$ is sufficiently small we have
\begin{equation}
\label{Ts}
|T(s)|\le L^{-\frac23\ell}\,.
\end{equation}

Let $\mathcal{S}$ denote the set of all scenarios. Now we want to calculate the number of scenarios. This is easy:
$$
\card\mathcal{S} \le m\cdot {\ell\choose 0.1\ell} \le \ell\cdot\bigg(\frac{10}{9}\bigg)^{0.9\ell} \cdot 10^{0.1\ell}\,.
$$

We just proved that the measure of the set of all $t\in(0,1)$ such that one has \eqref{Hopposite}
$$
\exists y\in R_t, : y_1\in [ \frac1{\sqrt{2}}, \,L^m] \,:\,\,\,\sum_{k=0}^{\ell} \dist (L^k\,y, \Z^2) \le \tau\,\ell
$$
can be estimated as
$$
\le \ell\cdot\bigg(\frac{10}{9}\bigg)^{0.9\ell} \cdot 10^{0.1\ell}\cdot L^{-\frac23\ell}\le L^{-\ell/2}\,.
$$

Proposition \ref{numbertheory} is proved. We indeed have very few  exceptional directions in the sense that on them $|\Pflt |$ is not uniformly smaller than $e^{-\tau\ell}$.

\end{proof}

Here is the second  claim of the subsection:
\begin{prop}
\label{Pshpest}
$$
t\in E\Rightarrow\int_{SSV(t)}|\Pshp|^2dx\leq C''K\,L^m.
$$
\end{prop}

We will see in Section \ref{structureofSSV} that for each $t$, $SSV(t)$ is contained in $C\cdot L^m$ neighborhoods of size $L^{n-m-\ell}$ around the complex zeroes $\lambda_j$ of $P_2$. 

Fix $t$. Let
\begin{equation}\label{Ijt}I_j=[\lambda_j-L^{n-m-\ell},\lambda_j+L^{n-m-\ell}],\end{equation}
\begin{equation}\text{where }SSV(t)\subseteq\bigcup_jI_j\end{equation}
Choose $j$ for which $\int_{I_j}|\Pshp|^2dx$ is maximized. Then
$$
\int_{SSV(t)}|\Pshp|^2dx\leq CL^m\int_{I_j}|\Pshp|^2dx\leq CL^m(L^{\ell+m-n})^2\int_{I_j}|\sum_{k=0}^{n-m-\ell} e^{i\alpha_jx}|^2.
$$
As $|I_j|\leq 2\cdot L^{n-m-\ell}$, so Lemma $\ref{CETSQ}$ and the definition of $E$ give us Proposition $\ref{Pshpest}$.

\bigskip

The estimate for $t\in \tilde{E}\setminus H$ follows. If $|E|\ge 1/K, K=L^{m/2}$, $|\tilde{E}|>1/2K$, and we also just proved that $|H|\le L^{-\ell/2}$, $\ell=\alpha\,m$ with large $\alpha$, we have a set $\mathcal{E}\subset \tilde{E}\setminus H$, $\mathcal{E}>1/4K$,
such that for every $t\in \mathcal{E}$
$$
\int_{SSV(t)} |P_1(r)|^2 \,dr \le L^{-\ell}\int_{SSV(t)} |\Pshp (r)|^2\,dr \le  C''K\,L^m\cdot L^{-\alpha m} \,.
$$

So we proved
\begin{equation}
\label{prop6}
\int_{SSV(t)} |P_1(r)|^2 \,dr \le c\, L^m
\end{equation}
with $c$ as small as we wish. In particular, Proposition \ref{P1onSSV}  is completely proved.

\section{Combinatorial part}
\label{combinatorics}

In this section, we show how Theorem $\ref{mainth1}$ follows from Theorem $\ref{mofE}$.

First, let us define
\begin{equation}\label{LthetaN}
\mathcal{L}_{\theta,N}:=proj_\theta\G_N.
\end{equation}

\begin{theorem}
\label{gooddir1}
Let $\beta>2$. (We used $\beta=3$ in the previous section). If $t\notin E$ (see definition $\eqref{Edef}$), then $|\mathcal{L}_{\theta, NK^{\beta}}| \le \frac{C}{K}.$
\end{theorem}

\begin{proof}
Let us use $\theta$ instead of $t$ and $x$ for the space variable on the non-Fourier side, since we do not use Fourier analysis in this proof. Fix $\theta$ and let $F:=A_K^*=\{x: f_N^*(x) \ge K\}$. We denote by $N_x$ the line orthogonal to direction $\theta$ and passing through $x$. We can call it needle at $x$. For every $x\in F$ there are at least $K$ discs of size $L^{-r}, r=r(x), r\le N$, intersecting $N_x$. Mark them. Run over all $x\in F$. Consider all marked discs. Consider all $L^{-N}$-discs that are sub-discs of marked ones. Call them ``green". Let $U$ be a family of green discs.

We want to show
\begin{equation}
\label{star}
\text{card}\, U \ge c\cdot K\,|F|\, L^N\,,
\end{equation}
\begin{equation}
\label{starstar}
|\text{proj}\,(\cup_{q\in U}q)| \le \frac{C}{ K}\,\text{card}\, U\,L^{-N}\,,
\end{equation}

Let $\phi:= \sum_{q\in U} \chi_q$. Then
$$
\int\phi\,dx = \text{card}\, U\, L^{-N}\,.
$$
Let $M$ denote  uncentered maximal function. To prove \eqref{starstar} it is enough to show that
$$
q\in U \Rightarrow  \text{proj}\,q\subset \{x: M\phi(x) >\frac{K}{C}\}\,,
$$
and then to use Hardy--Littlewood maximal theorem. But to prove this claim is easy. In fact, let $x\in \text{proj}\,q, q\in U$, then there exists $Q$--the maximal (by inclusion) marked disc containing $q$. Consider $I:= [x-10\,\ell(Q), x+10\,\ell(Q)]$. This segment contains the projections of at least $K$ disjoint discs $Q_1:=Q, Q_2,..., Q_K,...$, of the same sidelength, which intersect $N_{x_0}$, where $x_0$ is a point because of which $Q=Q_1$ was marked. (The reader should see that $x_0$ lies really well inside $I$.) So $I$ contains the projections of at least $\frac{\ell(Q)}{\ell(q)}\cdot K$ green triangles. Whence,
$$
\int_I\phi\,dx \ge \ell(q)\cdot \frac{\ell(Q)}{\ell(q)}\cdot K\ge \frac1{20} |I|\,K\,.
$$So
$$
M\phi(x)> \frac1{20}\,K\,.
$$
We proved \eqref{starstar}.

\medskip

Also we proved that $F\subset \{ x: M\phi (x) \ge \frac{K}{20}\}$. Therefore, by Hardy--Littlewood maximal theorem
$$
|F|\le |\{ x: M\phi (x) \ge \frac{K}{20}\}|\le \frac{C\,\int\phi}{K} = C\,\text{card}\,U\, L^{-N}\,K^{-1}\,.
$$
This is \eqref{star}.

\bigskip

Let us estimate $|\mathcal{L}_{\theta, N\, K^{\alpha}}|$ using \eqref{star} and \eqref{starstar}.
The first step:
$$
|\mathcal{L}_{\theta, N}| \le |\text{proj}\,(\cup_{q\in U} q)| + L^{-N} (L^N -\text{card}\, U)\le
$$
$$
\frac{C}{K}\text{card}\, U \,L^{-N} + (L^N-\text{card}\,U) L^{-N}\,.
$$
We do not touch the first term, but we improve the second term by using self-similar structure and going to step $2N$ (inside triangles which are not green there are ``green" discs of size $L^{-2N}$). They are just self-similar copies of the original green discs.
 Then
we have the second step:
$$
|\mathcal{L}_{\theta, N}| \le \frac{C}{K}\text{card}\, U \,L^{-N} + \text{the rest}\le
$$
$$
\frac{C}{K}\text{card}\, U \,L^{-N} + (L^N-\text{card}\,U) \frac{C}{K}\text{card}\, U \,L^{-2N} +(L^N-\text{card}\,U)^2\,L^{-2N}\,.
$$

Now we leave first two terms alone and having $(L^N-\text{card}\,U)^2$ triangles of size $L^{-2N}$ we find again ``green" discs inside each of those, now green triangles of size $L^{-3N}$. They are just self-similar copies of original green discs.

 Then
we have the third step:
$$
|\mathcal{L}_{\theta, 3N}| \le 
\frac{C}{K}\text{card}\, U \,L^{-N} + (L^N-\text{card}\,U) \frac{C}{K}\text{card}\, U \,L^{-2N} +\text{the rest} \le
$$
$$
\frac{C}{K}\text{card}\, U \,L^{-N} + (L^N-\text{card}\,U) \frac{C}{K}\text{card}\, U \,L^{-2N}+ (L^N-\text{card}\,U)^2 \frac{C}{K}\text{card}\, U \,L^{-2N}+
$$
$$
(L^N-\text{card}\,U)^3\,L^{-3N}\,.
$$
After the $l$-th step:
$$
|\mathcal{L}_{\theta, l\,N}| \le \frac{C}{K}\text{card}\, U \,L^{-N} (1 +(L^N-\text{card}\,U)L^{-N}+...
$$
$$
+(L^N-\text{card}\,U)^{l-1}L^{-(l-1)N}) + (L^N-\text{card}\,U)^{l}L^{-lN}\,.
$$
So
$$
|\mathcal{L}_{\theta, l\,N}| \le \frac{C}{K}\text{card}\, U \,L^{-N}  \frac{(1-(1-\frac{\text{card}\,U}{L^N})^l)}{(1-(1-\frac{\text{card}\,U}{L^N}))} +
$$
$$
 e^{-\frac{\text{card}\,U}{L^N} l}=: I +II\,.
$$
Notice that by \eqref{star} $II\le e^{-K|F|l} \le e^{-K}$ if the step $l$ is chosen to be $l=1/|F|\le K^{\beta}$. However, we always have $I\le \frac{C}{K}$. So Theorem \ref{gooddir1} is completely proved.
\end{proof}

From Theorems $\ref{mofE}$ and $\ref{gooddir1}$, it is not hard to get Theorem $\ref{mainth1}$.

\section{Putting $SSV(t)$ into a fixed number of intervals of correct size}
\label{structureofSSV}

Now we have to consider $P_{2,t}(r)= \phi_t(r)\phi_t(L^{-1}r)\cdot\dots\cdot \phi_t(L^{-m}r)$. We are interested in the set
$$
SSV(t):=\{ r\in [1,L^{m}]: \, |P_{2,t}(r)| \le L^{-Am^2}\}\,.
$$

We will be using so-called Turan's lemma:

\begin{lemma}
\label{turan}
Let $f(x)=\sum_{l=1}^L c_l e^{\lambda_l x}$, let $E\subset I$, $I$ being any interval.
Then
$$
\sup_I|f(x)|\le e^{\max|\Re \lambda_n|\,|I|} \bigg(\frac{A|I|}{|E|}\bigg)^L\sup_E|f(x)|\,.
$$
Here $A$ is an absolute constant.
\end{lemma}

In this form it is proved by F. Nazarov \cite{N}. 

Now let us consider any square $Q=[x'-1, x'+1]\times [-1,1]$.  We call $\frac12 Q$ the concentric square of half the size.

\begin{lemma}
\label{turanus}
With uniform constant $C$ depending only on $L$ one has
$$
\sup_Q |\phi_t(z)| \le C\, \sup_{\frac12 Q} |\phi_t(z)|\,.
$$
\end{lemma}

\begin{proof}

Let $z_0=x_0+iy_0$ is a point of maximum in the closure of $Q$. We first want to compare $|f(z_0)|$ and $|f(x_0)|$. Consider
$f_{x_0}(y) :=\phi_t(x_0+iy)$.  Notice that uniformly in $Q$ and $x_0$
$$
|f_{x_0}'(y) |\le C(L)\,.
$$
This means that $|f_{x_0}(y) |\ge \frac12 |f_{x_0}(0)|$ on an interval of uniform length $c(L)$.

Notice also that the exponents $\lambda_l(t), l=1,\dots, L,$ encountered in $\phi_t$ are all uniformly bounded.
Then applying Lemma \ref{turan}
we get
$$
|\phi_t(z_0)|=|f_{x_0}(y_0)|\le C'(L) |f_{x_0}(0)|   \,.
$$
Now consider $F(x)= \phi_t(x)$. We want to compare $F(x_0)=f_{x_0}(0) =\phi_t(x_0)$ with $\max_{[x'-\frac12,x'+\frac12]}|F(x)|$. By Lemma \ref{turan} we get again
$$
|f_{x_0}(0) |= |F(x_0)|\le \sup_{[x'-1, x'+1]} |F(x)| \le C''(L) \sup_{[x'-1/2, x'+1/2 ]} |F(x)| \le C''(L) \sup_{\frac12 Q} |\phi_t(z)|\,
$$
Combining the last two display inequalities we get Lemma \ref{turanus} completely proved.

\end{proof}

\begin{lemma}
\label{turanus1}
With uniform constant $C$ depending only on $L$  (and not on $m$) one has
$$
\sup_Q |\phi_t(L^{-k}z)| \le C\, \sup_{\frac12 Q} |\phi_t(L^{-k}z)|\,,\,k=0,\dots, m\,.
$$
\end{lemma}

The proof is exactly the same. We just use  $L^{-k}\lambda_l(t), l=1,\dots, L,$ encountered in $\phi_t(L^{-k}\cdot)$ are all uniformly bounded.

By complex analysis lemmas from Section \ref{lemmas} we know that
Lemma \ref{turanus1} implies that every $\frac12 Q$ has at most $M$ (depending only on $L$) zeros of $\phi_t(z)$. And if we denote them by $\mu_1,\dots,\mu_M$ then
\begin{equation}
\label{ssv}
\{x\in \frac12 Q\cap \R: |\phi_t(x)|\le L^{-M\ell}\} \subset \cup_{i=1}^M  B(\mu_i, L^{-\ell})\,.
\end{equation}

Consider $\mu_1,\dots, \mu_{S}$ being all zeros of $P_{2,t}$ in $[1/2, L^m+1]\times [1/2,1/2]$. By abovementioned lemmas from Section \ref{lemmas} and by Lemma \ref{turanus1} we get that
$$
S\le M(L)\,L^m\,.
$$
From \eqref{ssv} it is immediate that

\begin{equation}
\label{ssv1}
\{x\in [1, L^m]: \, |P_{2,t}(L^{-(n-m)}x)|= |\phi_t(x)\cdot\dots\cdot \phi_t(L^{-m}(x))|\le L^{-M\ell m}\} \subset \cup_{i=1}^{M\,L^m}  B(\mu_i, L^{-\ell})\,.
\end{equation}

Changing the variable $y=L^{n-m}x$ we get the structure of the set of small values used above during the proof of Proposition \ref{Pshpest}:
\begin{equation}
\label{SSV10}
SSV(t)\subset \cup_{i=1}^{C\,L^m} I_i\,,
\end{equation}
where each interval $I_i$ has the length $2\, L^{n-m-\ell}$.

\section{Some important standard lemmas. A bit of complex analysis}
\label{lemmas}

There are a few important lemmas which we have appealed to repeatedly. The first claim, Lemma \ref{CET}, uses the Carleson imbedding theorem. A stronger version, Lemma $\ref{CETSQ}$, uses general $H^2$ theory. Its importance lies in its ability to establish a key relationship between the level sets of $f_{n,t}$ and the $L^2$ norm of $\widehat{f_{n,t}}$. This is because the Fourier transform changes the centers of intervals into the frequencies of an exponential polynomial.

The second claim we split into Lemmas $\ref{schke1}$ and $\ref{schke2}$. Given a bounded holomorphic function on the disc, its supremum, and an interior non-zero value, these lemmas bound the number of zeroes and contain the set of small values within certain neighborhoods of these zeroes.

\subsection{{\bf In the spirit of the Carleson imbedding theorem}}
\begin{lemma}
\label{CET}
Let $j=1,2,...k$, $c_j\in\C$, $|c_j|=1$, and $\alpha_j\in\R$. Let $A:=\lbrace \alpha_j\rbrace_{j=1}^k$. Then
$$
\int_0^1{|\sum_{j=1}^k{c_je^{i\alpha_jy}}|^2dy}\leq C\,k\cdot \sup_{I\text{ a unit interval}}\card \{A\bigcap I\}\,.
$$
\end{lemma}

\begin{proof}
Let $A_1:=\{\mu= \alpha +i: \alpha\in A\}$. Let $\nu:=\sum_{\mu\in A_1} \delta_{\mu}$. This is a measure in $\C_+$. Obviously its Carleson constant
$$
\|\nu\|_C :=\sup_{J\subset \R,\, J\,\text{is an interval}}\frac{\nu(J\times [0,|J|])}{|J|}
$$ can be estimated as follows
\begin{equation}
\label{CETeq}
\|\nu\|_C \le 2\,\sup_{I\text{ a unit interval}}\card \{A\bigcap I\}\,.
\end{equation}

Recall that
\begin{equation}
\label{CETeq2}
\forall f\in H^2(\C_+)\,\,\int_{C_+} |f(z)|^2 \,d\nu(z) \le C_0\, \|\nu\|_C\|f\|_{H^2}^2\,,
\end{equation}
where $C_0$ is an absolute constant.
Now we compute
$$
\int_0^1{|\sum_{j=1}^k{c_je^{i\alpha_jy}}|^2dy}\leq e^2\int_0^1{|\sum_{j=1}^k{c_je^{i(\alpha_j+i)y}}|^2dy}\leq
$$
$$
 e^2\int_0^{\infty}{|\sum_{j=1}^k{c_je^{i(\alpha_j+i)y}}|^2dy} = e^2\int_{\R}|\sum_{\mu\in A_1}\frac{c_{\mu}}{x-\mu}|^2\,,
 $$
 where $c_{\mu} := c_j$ for $\mu= \alpha_j +i$. The last equality is by Plancherel's theorem.
 
 We continue
$$
\int_{\R}|\sum_{\mu\in A_1}\frac{c_{\mu}}{x-\mu}|^2 =\sup_{f\in H^2(C_+),\, \|f\|_2\le 1}\bigg|\langle f,\sum_{\mu\in A_1}\frac{c_{\mu}}{x-\mu}\rangle\bigg|^2=
$$
$$
 4\pi^2\sup_{f\in H^2(C_+),\, \|f\|_2\le 1}|\sum_{\mu\in A_1}c_{\mu}f (\mu)|^2\le C\,\card\{A_1\}\sup_{f\in H^2(C_+),\, \|f\|_2\le 1}\sum_{\mu\in A_1} |f(\mu)|^2 \le
 $$
 $$
 C\,\card\{A\}\sup_{f\in H^2(C_+),\, \|f\|_2\le 1}\int_{C_+}|f(z)|^2\,d\nu(z)\le 
 2C_0C\,\card\{A\}\,\sup_{I\text{ a unit interval}}\card \{A\bigcap I\}\,.
 $$
 This is by \eqref{CETeq1} and \eqref{CETeq}. The lemma is proved.

\end{proof}

\medskip

Now we are going to prove a stronger assertion by a simpler approach. This stronger assertion is what is used in the main part of the article.

\medskip

\begin{lemma}
\label{CETSQ}
Let $j=1,2,...k$, $c_j\in\C$, $|c_j|=1$, and $\alpha_j\in\R$. Let $A:=\lbrace \alpha_j\rbrace_{j=1}^k$. Then
Suppose
\begin{equation}
\label{sumch}
\int_{\R} (\sum_{\alpha\in A} \chi_{[\alpha-1, \alpha+1]}(x))^2\,dx \le S\,,
\end{equation}
Then there exists an absolute constant $C$
\begin{equation}
\label{sumexp}
\int_0^1 |\sum_{\alpha\in A} c_{\alpha} e^{i\alpha y}|^2\, dy \le C\,S\,.
\end{equation}
\end{lemma}

Of course, one can change variables and get:
\begin{cor}
Let $j=1,2,...k$, $c_j\in\C$, $|c_j|=1$, and $\alpha_j\in\R$. Let $A:=\lbrace \alpha_j\rbrace_{j=1}^k$, and let $\delta>0$.
Suppose
\begin{equation}
\int_{\R} (\sum_{\alpha\in A} \chi_{[\alpha-\delta, \alpha+\delta]}(x))^2\,dx \le S\,,
\end{equation}
Then there exists an absolute constant $C$
\begin{equation}
\label{sumexp2}
\int_a^{a+\delta^{-1}} |\sum_{\alpha\in A} c_{\alpha} e^{i\alpha y}|^2\, dy \le C\,S\,/{\delta^2}.
\end{equation}
\end{cor}

\noindent{\bf Remark.} Lemma \ref{CETSQ} is obviously stronger than Lemma \ref{CET}. In fact, let $S_0$ be the maximal number of points $A$ in any unit interval. Then
$$f(x) :=\sum_{\alpha\in A} \chi_{[\alpha-1,\alpha+1]}(x)\le 2S_0.$$
Now $\int_{\R} f^2(x) dx\le 4kS_0$, where $k$ as above is the cardinality of $A$. We can put now $S:= 4kS_0$, apply Lemma \ref{CETSQ} and get the conclusion of Lemma \ref{CET}. The proof of Lemma \ref{CETSQ} does not require the Carleson imbedding theorem. Here it is.

\begin{proof}
Using Plancherel's theorem we write
$$
\int_0^1|\sum_{\alpha\in A} c_{\alpha} e^{i\alpha\,y}\, dy|^2 \le e\int_0^1|\sum_{\alpha\in A} c_{\alpha} e^{i(\alpha+i)\,y}\, dy|^2 \le e\int_0^{\infty}|\sum_{\alpha\in A} c_{\alpha} e^{i(\alpha+i)\,y}\, dy|^2=
$$
$$
e\int_{\R} \bigg|\sum_{\alpha\in A}\frac{c_{\alpha}}{\alpha +i -x}\bigg|^2\,dx\,.
$$

Recall that
\begin{equation}
\label{CETeq1}
 H^2(\C_+)\,\,\text{is orthogonal to}\,\, \overline{H^2(\C_+)}
\end{equation}

Now we continue
$$
\int_{\R} \bigg|\sum_{\alpha\in A}\frac{c_{\alpha}}{\alpha +i -x}\bigg|^2\,dx\le 
$$
$$
 \int_{\R} \bigg|\sum_{\alpha\in A}\frac{c_{\alpha}}{\alpha +i -x}-\sum_{\alpha\in A}\frac{c_{\alpha}}{\alpha -i -x}\bigg|^2\,dx=
 $$
 $$
 \frac{\pi}{2}\int_{\R} \bigg|\sum_{\alpha\in A}c_{\alpha}P_1(\alpha-x)\bigg|^2\,dx\,,
 $$
 where $P_1$ is the Poisson kernel in the half-plane $C_+$ at hight $h=1$:
 $$
 P_h(x):=\frac1{\pi}\frac{h}{h^2 +x^2}\,.
 $$
 We continue by noticing that $P_1*\chi_{[\lambda-1,\lambda+1]}(x) \ge c\,P_1(\lambda-x)$ with absolute positive $c$. This is an elementary calculation, or, if one wishes, Harnack's inequality. Now we can continue
$$
\int_0^1|\sum_{\alpha\in A} c_{\alpha} e^{i\alpha\,y}\, dy|^2 \le\frac{\pi e}{2c}\int_{\R} \bigg|(P_1*\sum_{\alpha\in A}c_{\alpha}\chi_{[\alpha-1,\alpha+1]})(x)\bigg|^2\,dx\,.
$$
Now we use the fact that $f\rightarrow P_1*f$ is a contraction in $L^2(\R)$. So

$$
 \int_0^1|\sum_{\alpha\in A} c_{\alpha} e^{i\alpha\,y}\, dy|^2 \le\frac{\pi e}{2c}\int_{\R} |\sum_{\alpha\in A}c_{\alpha}\chi_{[\alpha-1,\alpha+1]}(x)|^2\,dx\le C\, S\,.
 $$
 The lemma is proved.

\end{proof}

\subsection{{\bf A Blaschke estimate}}
\begin{lemma}\label{schke1}
Let $D$ be the closed unit disc in $\C$. Suppose $\phi$ is holomorphic in an open neighborhood of $D$, $|\phi(0)|\geq 1$, and the zeroes of $\phi$ in $\frac{1}{2}D$ are given by $\lambda_1,\lambda_2,...,\lambda_M$. Let $C=||\phi||_{L^\infty(D)}$. Then $M\leq log_2(C).$
\end{lemma}
\begin{proof}
Let 
$$
B(z)=\prod_{k=1}^M{\frac{z-\lambda_k}{1-\bar{\lambda_k}z}}.
$$
 Then $|B|\leq 1$ on $D$, with $=$ on the boundary. If we let $g:=\frac{\phi}{B}$, then $g$ is holomorphic and nonzero on $\frac{1}{2}D,$ and $|g(e^{i\theta})|\leq C$ $\forall\theta\in [0,2\pi]$. Thus $|g(0)|\leq C$ by the maximum modulus principle. So we have $$C\geq |g(0)|=\frac{|\phi(0)|}{|B(0)|}\geq\prod_{k=1}^M{\frac{1}{|\lambda_k|}}\geq 2^M.$$
\end{proof}

\begin{lemma}\label{schke2}
In the same setting as Theorem $\ref{schke1}$, the following is also true for all $\delta\in (0,1/3)$: $\lbrace z\in\frac{1}{4}D:|\phi|<\delta\rbrace\subseteq\bigcup_{1\leq k\leq M} B(\lambda_k,\e)$, where $$\e:=\frac{9}{16}(3\delta)^{1/M}\leq\frac{9}{16}(3\delta)^{1/log_2(C)}.$$
\end{lemma}
\begin{proof}
Let $\delta\in (0,1/3)$, and let $z\in\frac{1}{4}D$ such that $|z-\lambda_k|>\e\,\,\forall k$. Note that $g$ is harmonic and nonzero on $\frac{1}{2}D$ with $|g(0)|\geq 2^M$. Thus Harnack's inequality ensures that $|g|\geq\frac{1}{3}2^M$ on $\frac{1}{4}D$, so there 
$$
|\phi(z)|\geq |g(z)B(z)| \geq\frac{1}{3}2^M\prod_{k=1}^M{|\frac{z-\lambda_k}{1-\bar{\lambda_k}z}|}\geq(\frac{16\e}{9})^M\frac{1}{3}=\delta.
$$
 We can conclude the proof by the contrapositive.
\end{proof}

\section{Combinatorial theorem}
\label{combi}
For this section, regard the set $E$ from Section $\ref{sec:fourier}$ as parameterized by $\theta$, and use the variable $x$ instead of $s$ on the non-Fourier side, since we will not work on the Fourier side at all during this section.

\begin{theorem}
\label{combin}

 Let $\theta\in E$. Then
 $$
 \max_{n: 0\le n\le N} \|f_{n,\theta}\|^2_{L^2(\R)}\le C\, K\,.
 $$
\end{theorem}

To prove this we first need the following claim, which is the main combinatorial assertion of this article. It repeats the one in \cite{NPV} but we give a slightly different proof.

We fix a direction $\theta$, we think that the line $\ell_\theta$ on which we project is $\R$. If $x\in \R$ then by $N_x$ we denote the line orthogonal to $\R$ and passing through point $x$, we call $N_x$  a needle. By $F_L$ we denote $\{x\in\R: f^*_N(x) :=\max_{0\le n\le N} f_{n,\theta}(x) >L\}$ (also known as $A_L^*$).

\begin{theorem}
\label{combinlemma}
There exists an absolute constant $C$ such that for any large $K$ and $M$
\begin{equation}
\label{ssim1}
|F_{4KM}|\le C\, K\, |F_K|\cdot |F_M|\,.
\end{equation}
\end{theorem}

\begin{proof}
This will be a proof by greedy algorithm. First choose $y\in F_{4K}$ and consider needle $N_y$ and discs of certain size $L^{-j_y}, j_y\le N$ intersecting $N_y$. Consider any family of this sort having more than $4K$ elements. Fix such a family. We will ``fathorize" it, i.e. we consider the father of each element in the family. Two things may happen: 1) there are more than $4K$ distinct fathers; 2) number of fathers is at most $4K$. In the latter case the number of fathers is at least $2K$. In fact, we slash the number of elements by fathorizing, but not more than by factor of $1/2$. If the first case happens fathorize again, do this till we get to the second case. 

After doing this procedure with all $x\in F_{4K}$ and all families of cardinality bigger than $4K$ of equal size discs  intersecting needle $N_x$ we come to some awfully complicated set of discs. But we will consider now maximal-by-inclusion discs of this family, the family of these maximal discs is called $\mathcal{F}_0$.

\medskip

Choose disc $Q_{00}\in \mathcal{F}_0$ such that its sidelength $\ell(Q_{00})$  is maximal possible in $\mathcal{F}_0$. It is very important to notice that $\mathcal{F}_0$ contains at least $2K-1$ discs of the same size as $Q_{00}$ pierced by a needle $N_{y_0}$. This is because of maximality of the lengthsize, the stack pierced by $N_{y_0}$ could not be eaten up even partially by bigger in size discs from some other stack. So let us call by $Q_{01},..., Q_{02K-1},..., Q_{0S}$, $S\ge 2K-1$. They are of the same size as $Q_{00}$ and all intersect a certain needle $N_{y_0}$. 

\medskip

Denote $$ I_0=\text{proj}\,Q_{00}\,.$$

\medskip

Consider all $q\in \mathcal{F}_0$ such that $$\text{proj}\,q\cap 20\,I_0\neq\emptyset\,.$$
Call them $\mathcal{F}(Q_{00})$. Of course $\ell(q) \le \ell(Q_{00})$. For every such $q$ consider
a Cantor square $Q$, $q\subset Q$, such that $\ell(Q)=\ell(Q_{00})$. Such $Q$'s form family $\tilde{\mathcal{F}}(Q_{00})$.

\begin{lemma}
\label{cl1}
For every $y\in \R$ the needle $N_y$ intersects at most $4K$ discs of the family $\tilde{\mathcal{F}}(Q_{00})$.
\end{lemma}

\begin{proof}

Suppose contrary. Then $N_y$ intersects more than $4K$ of discs from $\tilde{\mathcal{F}}(Q_{00})$. So $y\in F_{4K}$, and our pierced family is one of those which we considered at the beginning. It can be fathorized. Then the square of size $\ge 2\, \ell(Q_{00})$ will be present in $\mathcal{F}_0$. Contradiction with maximality of length.

\end{proof}

\begin{lemma}
\label{cl2}
$\text{card}\,\tilde{\mathcal{F}}(Q_{00})\le 88\,K\,.$
\end{lemma}

\begin{proof}
$$
\text{card}\,\tilde{\mathcal{F}}(Q_{00})\cdot \ell(Q_{00}) = \sum_{Q\in \tilde{\mathcal{F}}(Q_{00})}\ell(Q) \le 
$$
$$
\int_{22 I_0} \text{card}\, \{Q\in \tilde{\mathcal{F}}(Q_{00}): Q\cap N_y\neq \emptyset\}\,dy\le
$$
$$
4K\cdot 22\ell(Q_{00})\,.
$$
This is by Lemma \ref{cl1}.

\end{proof}

\begin{lemma}
\label{cl3}
There exists an interval $J_0\subset I_{y_0}$ such that $|J_0| \ge c\cdot |I_0|$ with a certain absolute positive $c$. And $J_0\subset F_{K}$.
\end{lemma}

\begin{proof}
We already noticed that $Q_{00}, Q_{01},..., Q_{02K-1}$ intersect needle $N_{y_0}$. Then at least half of them have their center of symmetry to the right of $N_{y_0}$, or at least half of them have their center of symmetry to the left of $N_{y_0}$. Assume that the first case occurs. Then the segment $[y_0, c\cdot \ell(Q_{00})]$ obviously is contained in $F_K$.

\end{proof}

\begin{lemma}
\label{cl4}
$|F_{4KM}\cap 20I_0| \le C\, K\, \ell(Q_{00})|F_M|=C\,K\,|I_0||F_M|\,.$
\end{lemma}

\begin{proof}
Of course $F_{4KM}\subset F_{4K}$. For $y\in F_{4KM}\cap 20I_0$ the whole family of small discs whose quantity is $> 4KM$ intersecting $N_y$ will be inside one of those $Q\in \tilde{\mathcal{F}}(Q_{00})$, whose number is at most $88K$ by Lemma \ref{cl2}. Let us enumerate $Q^1,..., Q^s$, $s\le 88K$ elements of $\tilde{\mathcal{F}}(Q_{00})$. So there exists $i=1,...,s$ such that
$$
y\in\text{dilated copy of}\,F_M\,\text{in}\,\text{proj}\,Q^i\,.
$$
Hence
$$
F_{4KM}\cap 20 I_0\subset \cup_{i=1}^{88K}\text{dilated copy of}\,F_M\,\text{in}\,\text{proj}\,Q^i\,.
$$
So
$$
|F_{4KM}\cap 20 I_0|\le \sum_{i=1}^{88K}\ell(Q^i)|F_M| \le 88K\, \ell(Q_{00})|F_M|\,.
$$

\end{proof}

\begin{lemma}
\label{cl5}
$|F_{4KM}\cap 20I_0| \le 88 c^{-1} K|F_m|\cdot |J_0|\,.$
\end{lemma}

Now we want to repeat all steps for $F_{4K}^0:= F_{4K} \setminus 20 I_0$. So we fathorize discs pierced by needles $N_x$, $x\in F_{4K}^0$. As before we get families $\mathcal{F}_1$, maximal sidelength triangle $Q_{11}$, families $\mathcal{F}(Q_{11})$, $\tilde{\mathcal{F}}(Q_{11})$. Notice that $\mathcal{F}_1<\mathcal{F}_0$ in the sense that for every $q\in \mathcal{F}_1$ there exists $q\in\mathcal{F}_0$ such that $q$ is contained in $Q$. It is also clear that 
$$
\ell(Q_{11})\le \ell(Q_{00})\,.
$$
Obviously $Q_{00}, Q_{01},...$ are not in $\mathcal{F}_1$, their projections even do not intersect $\R\setminus 20 I_0$.

There are at least $2K-1$ brothers of $Q_{11}$: $Q_{12},...,Q_{1 2K-1},...$ in $\mathcal{F}_1$ such that
they are of the same size $\ell(Q_{11})$ and they (and $Q_{11}$) intersect the same needle $N_{y_1}$, $y_1\in \R\setminus 20I_0$.
This is again the maximality of the sidelength among $\mathcal{F}_1$ discs. Let $I_1:= \text{proj}\, Q_{11}$.
Notice that
$$
I_1 \cap I_0=\emptyset\,.
$$
In fact, $y_1 \in I_1, y_1\notin 20I_0$, $Q_{11}$ size is much smaller than $20|I_0|$. We consider all $q\in \mathcal{F}_1$ such that
$$
\text{proj}\, q\cap (20I_1\setminus 20I_0) \neq \emptyset\,.
$$ 
Call this family $\mathcal{F}(Q_{11})$. For every $q\in \mathcal{F}(Q_{11})$ consider Cantor disc $Q$ containing $q$ and of the size $\ell_1=\ell(Q_{11})$. Maximal-by-inclusion among such $Q$'s form $\tilde{\mathcal{F}}(Q_{11})$.

\begin{lemma}
\label{cl1prime}
For any $y\in R\setminus 20I_0$, $N_y$ intersects at most $4K$ discs of $\tilde{\mathcal{F}}(Q_{11})$.
\end{lemma}

\begin{proof}
Suppose contrary. Then there exists $y_1'\in F_{4K}\cap (\R\setminus 20I_0)$, and a subfamily of  $\tilde{\mathcal{F}}(Q_{11})$ of cardinality bigger than $4K$ intersects $N_{y_1'}$. It can be fathorized. Then discs of size $\ge 2\ell(Q_{11})$ would belong to $\mathcal{F}_1$. This contradicts the maximality of $\ell(Q_{11})$.

\end{proof}

\begin{lemma}
\label{cl11}
For any $z\in \R$, $N_z$ intersects at most $8K$ discs of $\tilde{\mathcal{F}}(Q_{11})$.
\end{lemma}

\begin{proof}
Suppose contrary. Then there exists $z\in F_{4K}$, and a subfamily of  $\tilde{\mathcal{F}}(Q_{11})$ of cardinality bigger than $4K$ intersects $N_{z}$. Now there is an end-point of $20I_1\setminus 20I_0$ (call it $a$), which is closest to $z$. Let it be on the right of $z$. Then another end-point is also on the right but farther away. As every triangle from the family has a) $z$ in its projection, and b) a certain point to the right of $a$ in its projection (their projections intersect $20I_1\setminus 20I_0$--by definition), then all of them have $a$ in its projection. Let us be lavish and say that $50$ percent of them have $a$ in their projection (the fact is that it is not lavishness, it is necessity: next step will be to consider in the future $20I_2\setminus (20I_0\cup 20I_1)$, and their can be $2$ closest points to $z$: one on the left, say, $b$, and one on the right, say, $a$, and we can guarantee that $50$ percent of our discs have either $b$ or $a$ in their projections simultaneously). We use the previous Lemma \ref{cl1prime}, and get that this $5)$ percent is $\le 4K$. So we are done.

\end{proof}

\begin{lemma}
\label{cl2prime}
$\text{card} \, \tilde{\mathcal{F}}(Q_{11}) \le 172 K\,.$
\end{lemma}

\begin{proof}
$$
\text{card}\,\tilde{\mathcal{F}}(Q_{11})\cdot \ell(Q_{11}) = \sum_{Q\in \tilde{\mathcal{F}}(Q_{11})}\ell(Q) \le 
$$
$$
\int_{22 I_1} \text{card}\, \{Q\in \tilde{\mathcal{F}}(Q_{11}): Q\cap N_y\neq \emptyset\}\,dy\le
$$
$$
8K\cdot 22\ell(Q_{11})\,.
$$
This is by Lemma \ref{cl1}.

\end{proof}

\begin{lemma}
\label{cl3prime}
There exists an interval $J_1\subset I_1$, $|J_1|\le c\cdot |I_1|$, such that $J_1\subset F_K$.
\end{lemma}

\begin{proof}
The same proof as for Lemma \ref{cl3}.
\end{proof}

\begin{lemma}
\label{cl4prime}
$|F_{4KM}^0\cap 20 I_1| \le C\, K\, \ell(Q_{11}\le C|, K\, |I_1|\,.$
\end{lemma}

\begin{proof}

The same proof as for Lemma \ref{cl4}.
\end{proof}

Combining Lemmas \ref{cl3prime}, \ref{cl4prime} we get

\begin{lemma}
\label{cl5prime}
$|F_{4KM}^0\cap 20 I_1| \le  C\, c^{-1}\, K\, |J_1|\,.$
\end{lemma}

We continue by introducing 
$$
F^1_{4KM}= F_{4KM} \setminus (20 I_0\cup 20 I_1)\,.
$$
We repeat the whole procedure. There will be $I_2$, $J_2\subset I_2 \cap F_K, |J_2|\ge c\cdot |I_2|$:
$$
I_2 \cap (I_1\cup I_0) =\emptyset\,,
$$
$$
|F_{4KM} \cap 20 I_2| \le Cc^{-1}K|J_2||F_M|\,,
$$
et cetera. 

Finally,
$$
|F_{4KM}|\le |F_{4KM}\cap 20 I_0|+ |(F_{4KM}\setminus 20 I_0)\cap 20 I_1|+...+|(F_{4KM}\setminus 20 I_0\cup 20I_1\cup....20I_{j-1})\cap 20 I_j|+...\le
$$
$$
C'\,K\,|F_M|\sum_{j=0}^{\infty}|J_j| \le C'\,K\,|F_M|\,|F_K|\,.
$$
We are done with Theorem \ref{combinlemma}.
\end{proof}

\medskip

Now we can prove Theorem \ref{combin}.

\begin{proof}
Let $E_j:=\{x: f_{n,\theta}(x) >(4K)^{j+1}\}$, $j=0,1,... .$.
We know by Theorem \ref{combinlemma} that
$$
|E_j| \le (CK)^j|E_0|^{j+1}\,.
$$
Hence,
$$
\int f_{n, \theta}(x)^2\,dx \le 4K \int f_{n,\theta}(x)\,dx + \sum_{j+0}^{\infty}\int_{E_j\setminus E_{j+1}}f_{n, \theta}(x)^2\,dx \le
$$
$$
4CK + (4K)^{j+2} \,(CK)^j|E_0|^{j+1}\,.
$$
If $|\{x: f^*_N(x) >K\}| \le 1/K^{2+\tau}$ then for all $n\le N$ we can immediately read the previous inequality as 
$$
\int f_{n, \theta}(x)^2\,dx  \le C(\tau) \,K\,.
$$

\end{proof}

\section{Discussion}
\label{discu}

\subsection{{\bf Difficulties for more general self-similar sets}}
\label{general sets}

The reason we were able to prove the stronger estimate for the Sierpinski gasket is exactly given by \eqref{keyobs}. It is a quantified version of the fact that the three-term sum $\varphi(z)=1+e^{iz}+e^{itz}$ is zero if and only if the summands are $e^{2j\pi i}$, $j=0,1,2$, and that for such $z$, $\varphi(3^kz)=3$ for all integers $k\geq 1$. An alternate argument using this fact in this form is employed in \cite{BV1}. Both versions of this fact we call by the general term ``analytic tiling".

But there cannot be  such a thing in the general case.  Suppose we had $5$ self-similarities, and that for for some direction $\theta$, we had $\phi_\theta (x_0)= 1 + (-i) + i + e^{2\pi i/3} + e^{4\pi i/3}=0$. Then clearly, taking fifth powers of the summands results in another zero with exactly the same summands, in complete and utter contrast to the three-point case. Similar examples using partitions into relatively prime roots of unity exist for numbers other than $5$.

 Though perhaps there is some hope that for arbitrary sets, some other form of analytic tiling occurs for typical directions in the arbitrary case (with small measure of exceptional directions). Ergodic theory may be of central importance. For example, if one considers $\K_n$ as in \cite{NPV}, one gets $\varphi_\theta(z)=1+e^{i\pi z}+e^{i\lambda z}+e^{i(\lambda +\pi)z}$, which has the zero $z=1$. Then $\varphi(4^k)=2(1+cos(4^k\lambda))$ for $k>0$. $\lambda$ depends continuously on $\theta$, and for fixed $\lambda$ such an ergodic sampling results in a sequence $a_k:=\varphi(4^k)$, and either:
 
1: $a_k$ is eventually periodic and non-zero,

2: $a_k$ takes values other than $4$ only finitely often,

or 3 (the case for almost every $\lambda$): $4^k\lambda$ mod $2\pi$ evenly samples $[0,2\pi]$ over the long term, with long-term average $\frac1{N}\sum_{k=1}^Na_k\to 2$ as $N\to\infty$.


  \bibliographystyle{amsplain}

\end{document}